\documentclass[a4paper]{amsart}
\pdfoutput=1

\usepackage{amsthm}
\usepackage{amssymb}
\usepackage{dsfont}
\usepackage{enumerate}
\usepackage{enumitem}
\usepackage{stmaryrd}
\usepackage[mathcal]{euscript}
\usepackage{tikz-cd} 
\usepackage[all]{xy}
\SelectTips{cm}{}
\usepackage{microtype}
\usepackage{mathtools}

\usepackage[top=35mm, bottom=35mm, outer=30mm, inner=30mm, marginparwidth=20mm]{geometry}

\usepackage{hyperref}
\hypersetup{%
	bookmarksnumbered=true,%
	colorlinks=true,%
	linkcolor=[RGB]{0, 0, 153},
	citecolor=[RGB]{0,153,51},      
	urlcolor=[RGB]{153,0,0},
	filecolor=blue,%
	menucolor=blue,%
	bookmarksopen=true,%
	bookmarksdepth=2,%
	pageanchor=true}

\usepackage[nameinlink,capitalize]{cleveref}
\crefname{equation}{}{}
\crefname{lemma}{Lemma}{Lemmata}

\usepackage[sorting=nyt, backend=bibtex, style=alphabetic, maxnames=50]{biblatex}
\renewbibmacro{in:}{}
\makeatletter
\@namedef{subjclassname@2020}{\textup{2020} Mathematics Subject Classification}
\makeatother

\numberwithin{equation}{section}

\theoremstyle{plain}
\newtheorem{theorem}[equation]{Theorem}
\newtheorem{proposition}[equation]{Proposition}
\newtheorem{lemma}[equation]{Lemma}

\theoremstyle{definition}
\newtheorem{definition}[equation]{Definition}

\newtheorem*{claim}{Claim}

\theoremstyle{remark}
\newtheorem{remark}[equation]{Remark} 

\newtheorem*{ack}{Acknowledgements}
\newtheorem*{conventions}{Conventions}

\newcommand*{\intref}[2]{\def\tmp{#1}\ifx\tmp\empty\hyperref[#2]{\ref*{#2}}\else\hyperref[#2]{#1~\ref*{#2}}\fi}

\newcommand{\Aut}{\operatorname{Aut}}

\newcommand{\CH}{\operatorname{CH}}
\newcommand{\ch}{\operatorname{ch}}

\newcommand{\Coh}{\operatorname{Coh}}

\renewcommand{\dim}{\operatorname{dim}}

\newcommand{\Ext}{\operatorname{Ext}}

\newcommand{\Hom}{\operatorname{Hom}}

\newcommand{\Pic}{\operatorname{Pic}}

\newcommand{\sheafhom}{\mathop{\mathcal{H}\! \mathit{om}}\nolimits}
\newcommand{\sheafext}{\mathop{\mathcal{E}\! \mathit{xt}}\nolimits}

\newcommand{\Supp}{\operatorname{Supp}}

\newcommand{\mcC}{\mathcal{C}}

\newcommand{\mcH}{\mathcal{H}} 

\newcommand{\mcL}{\mathcal{L}}
\newcommand{\mcO}{\mathcal{O}}
\newcommand{\mcP}{\mathcal{P}}

\newcommand{\mcU}{\mathcal{U}}

\newcommand{\sfD}{\mathsf D}

\newcommand{\bbC}{\mathbb C}

\newcommand{\bbP}{\mathbb P}
\newcommand{\bbQ}{\mathbb Q} 
 
\newcommand{\bbZ}{\mathbb Z}

\title[Autoequivalences of Blow-Ups of Minimal Surfaces]{Autoequivalences of Blow-Ups of Minimal Surfaces}

\author[Xianyu~Hu]{Xianyu Hu}
\address{Fakult\"at f\"ur Mathematik,
	Technische Universit\"at M\"unchen, D-85747 Garching bei M\"unchen, Germany}
\email{xianyu.hu@tum.de}

\author[Johannes~Krah]{Johannes Krah}
\address{Fakult\"at f\"ur Mathematik, Universit\"at Bielefeld, D-33501 Bielefeld, Germany
}
\email{jkrah@math.uni-bielefeld.de}

\thanks{The research of both authors
	was funded by the Deutsche Forschungsgemeinschaft (DFG, German Research Foundation) -- Project-ID 491392403 -- TRR 358.
	Further, the research of X.H.\ was funded by the Deutsche Forschungsgemeinschaft (DFG, German Research Foundation) -- 516701553, and the research of J.K.\ was funded by the Deutsche Forschungsgemeinschaft (DFG, German Research Foundation) under Germany's Excellence Strategy -- EXC-2047/1 -- 390685813.}

\begin{document}
	
	\begin{abstract} 
		Let $X$ be the blow-up of $\bbP^2_\bbC$ in a finite set of very general points.
		We deduce from the work of Uehara \cite{uehara_a_trichotomy_for_the_autoequivalence_groups_on_smooth_projective_surfaces} that $X$ has only standard autoequivalences, no nontrivial Fourier--Mukai partners, and admits no spherical objects.
		If $X$ is the blow-up of $\bbP^2_\bbC$ in $9$ very general points, we provide an alternate and direct proof of the corresponding statement.
		Further, we show that the same result holds if $X$ is a blow-up of finitely many points in a minimal surface of nonnegative Kodaira dimension which contains no $(-2)$-curves.
		Independently, we characterize spherical objects on blow-ups of minimal surfaces of positive Kodaira dimension.
	\end{abstract}
	
	\keywords{Derived Categories, Autoequivalences, Rational Surfaces, Spherical Objects}
	\subjclass[2020]{14F08, 14J26}

	
	\maketitle
	
	\section{Introduction}
	Let $X$ be a smooth projective variety over the complex numbers and denote by $\sfD^b(X)$ the bounded derived category of coherent sheaves on $X$.
	If the canonical bundle $\omega_X$ is ample or anti-ample, then, by Bondal--Orlov \cite{bondal_orlov_reconstruction_of_a_variety_from_the_derived_category_and_groups_of_autoequivalences}, the group of autoequivalences $\Aut(\sfD^b(X))$ only consists of so-called \emph{standard autoequivalences}, i.e.\ 
	$$ \Aut(\sfD^b(X)) = \Pic(X)\rtimes \Aut(X)\times \bbZ[1].$$
	In general, the standard autoequivalences $\Pic(X)\rtimes \Aut(X)\times \bbZ[1]$ form a subgroup of $\Aut(\sfD^b(X))$ and $\sfD^b(X)$ often admits of non-standard autoequivalences, see, e.g., \cite{orlov_derived_categories_of_coherent_sheaves_on_abelian_varieties_and_equivalences_between_them} for the case of abelian surfaces and \cite{bayer_bridgeland_derived_automorphism_groups_of_k3_surfaces_of_picard_rank_1} for the case of $K3$ surfaces of Picard rank $1$.
	A natural source for non-standard autoequivalences are so-called \emph{spherical twists} \cite{seidel_thomas_braid_group_actions_on_derived_categories_of_coherent_sheaves}.
	
	In contrast to the case of varieties with trivial canonical class, a spherical object on a variety with nontrivial and non-torsion canonical class has to be supported on a proper closed subset, see \cref{lem:spherical_obj_supported_on_curve}.
	If $X$ is a certain toric surface \cite[Thm.~1, Thm.~2]{broomhead_ploog_autoequivalences_of_toric_surfaces} or a surface of general type whose canonical model has at worst $A_n$-singularities \cite[Thm.~1.5]{ishii_uehara_autoequivalences_of_derived_categories_on_the_minimal_resolutions_of_an_singularities_on_surfaces}, then $\Aut(\sfD^b(X))$ is generated by standard equivalences and spherical twists.
	
	In the first part of this paper, we focus on rational surfaces $X$ which are blow-ups of $\bbP^2_\bbC$ in $n$ very general points.
	By a result of de~Fernex, recalled in \cref{prop:negative_curves_de_fernex}, such a surface $X$ does not contain any $(-2)$-curve.
	Motivated by the results of \cite{ishii_uehara_autoequivalences_of_derived_categories_on_the_minimal_resolutions_of_an_singularities_on_surfaces}, it is reasonable to expect that the absence of $(-2)$-curves implies the absence of spherical objects.
	Indeed, utilizing the above result of de~Fernex, we deduce the following \cref{thm:main_thm} from the work of Uehara, see \cref{sec:general_case}.
	\begin{theorem}\label{thm:main_thm}
		Let $X$ be the blow-up of $\bbP^2_\bbC$ in $n$ very general points. Then the following statements hold:
		\begin{enumerate}
			\item \label{item:autoequvialences_are_standard} Any autoequivalence of $X$ is standard, i.e.\ $\Aut(\sfD^b(X)) = \Pic(X)\rtimes \Aut(X)\times \bbZ[1]$.
			\item \label{item:no_fm_partners} If $Y$ is a smooth projective variety such that $\sfD^b(X) \cong \sfD^b(Y)$, then $X \cong Y$.
			\item \label{item:no_spherical_objects} There exists no spherical object in $\sfD^b(X)$.
		\end{enumerate}
	\end{theorem}
	
	By \cite[Cor.~4.4]{favero_reconstruction_and_finiteness_results_for_fm_partners}, for any smooth projective variety $X$ we have that \labelcref{item:autoequvialences_are_standard} implies \labelcref{item:no_fm_partners}.
	Moreover, if $X$ is a smooth projective variety of dimension $\geq 2$ with nontrivial and non-torsion canonical class, then \labelcref{item:autoequvialences_are_standard} implies \labelcref{item:no_spherical_objects}.
	Indeed, arguing as in \cref{lem:spherical_obj_supported_on_curve}, a spherical object $S$ on such a variety $X$ has to be supported on a proper closed subvariety.
	By \cite[Ex.~8.5\,(ii)]{huybrechts_fourier_mukai_transforms_in_algebraic_geometry}, the spherical twist $T_S$ associated to $S$ satisfies $T_S(S) = S[1-\dim X]$ and $T_S(k(x))=k(x)$ for any point $x\in X \setminus \Supp (S)$. Thus, $T_S$ is a non-standard autoequivalence.
	
	Note that \cref{thm:main_thm} follows from \cite{bondal_orlov_reconstruction_of_a_variety_from_the_derived_category_and_groups_of_autoequivalences} if $n\leq 8$ as $X$ is a del~Pezzo surface in this case. In fact, for $n\leq 8$ it is sufficient to require the blown up points to be in general position.
	
	In \cref{sec:9_points} we give an elementary proof of \cref{thm:main_thm} in the case $n=9$ by arguing that for all integral curves $C \subseteq X$ the restriction $K_X\vert_C$ is nonzero in $\CH^1(C)_\bbQ$.
	
	In \cref{se:surfaces_nonneg_kod_dim} we consider blow-ups $X$ of minimal surfaces $Y$ of nonnegative Kodaira dimension.
	In contrast to the case of rational surfaces, $(-2)$-curves on $X$ are strict transforms of $(-2)$-curves on $Y$, see \cref{prop:-2_curves_obj_in_blow_up_nef_canonical}.
	Thus, using \cite{uehara_a_trichotomy_for_the_autoequivalence_groups_on_smooth_projective_surfaces}, we obtain
	\begin{theorem}[\cref{thm:short_version_blow_up_nonneg_kod_dim}]
		\label{thm:thm_2_intro}
		Let $Y$ be a minimal
		surface of 
		nonnegative Kodaira dimension and let $X$ be the blow-up of $Y$ in a nonempty finite set of points.
		Assume $Y$ contains no $(-2)$-curves, e.g.\ $Y$ has Kodaira dimension $1$ and the elliptic fibration of $Y$ has only irreducible fibers.
		Then $\sfD^b(X)$ admits only standard autoequivalences, i.e.
		$$\Aut(\sfD^b(X)) =  \Pic(X) \rtimes \Aut(X) \times \bbZ[1].$$
	\end{theorem}
	As outlined above, \cref{thm:thm_2_intro} implies that such $X$ has no Fourier--Mukai partners and $\sfD^b(X)$ does not contain spherical objects.
	
	In \cref{prop:spherical_obj_in_blow_up_nef_canonical} we characterize spherical objects on blow-ups $X$ of minimal surfaces $Y$ of positive Kodaira dimension: An object in $\sfD^b(X)$ is spherical if and only if it is the pullback of a spherical object in $\sfD^b(Y)$ whose support is disjoint from the exceptional locus of $X \to Y$.
	If $Y$ is a minimal surface of Kodaira dimension $1$ whose elliptic fibration has only irreducible fibers, this characterization combined with the results of \cite{uehara_autoequvialences_of_derived_categories_of_elliptic_surfaces_with_non_zero_kodaira_dimension} gives an alternate proof that $\sfD^b(X)$ does not contain spherical objects, see \cref{rmk:spherical_objects_base_no_spherical_nef_canonical}.
	\begin{ack}
		We thank Hokuto Uehara and Charles Vial for useful comments on an earlier draft of this paper and we thank Gebhard Martin for helpful discussions regarding elliptic surfaces.
		Further, we thank the anonymous referee for helpful comments.
	\end{ack}
	\begin{conventions}
		The term \emph{surface} always refers to a smooth projective $2$-dimensional variety over~$\bbC$.
		For a variety $X$, we denote by $\CH^*(X)$ (resp.\ $\CH_*(X)$) the Chow groups of algebraic cycles modulo rational equivalence with integer coefficients graded by codimension (resp.\ dimension).
		We denote by $\CH^*(X)_\bbQ \coloneqq \CH^*(X) \otimes_\bbZ \bbQ$ the Chow groups with rational coefficients.
		A $(-k)$-curve $C$ in a surface $S$ is an integral smooth rational curve $C$ with self-intersection number $-k$.
		The term ``$n$ general points in $\bbP^2_\bbC$'' means that there exists a Zariski open subset $U\subseteq (\bbP^2_\bbC)^{n}$ such that for any $(p_1, \dots, p_n)\in U$ [...] holds.
		The term ``$n$ very general points in $\bbP^2_\bbC$'' means that there exist countably many Zariski open subset $U_i\subseteq (\bbP^2_\bbC)^{n}$ such that for any $(p_1, \dots, p_n)\in \bigcap_i U_i$ [...] holds.
	\end{conventions}

	\section{Preliminary Observations}
	Let $X$ be a smooth projective variety.
	The support of an object $F \in \sfD^b(X)$ is by definition the closed subvariety
	$$\Supp(F)\coloneqq \bigcup_{i \in \bbZ} \Supp(\mcH^i(F))\subseteq X$$
	endowed with the unique reduced closed subscheme structure.
	If $F$ is a simple object, i.e.\ $\Hom(F,F)=\bbC$, then $\Supp(F)$ is connected; see, e.g., \cite[Lem.~3.9]{huybrechts_fourier_mukai_transforms_in_algebraic_geometry}.
	\begin{definition}
		An object $S \in \sfD^b(X)$ is called \emph{spherical} if 
		\begin{align*}
			\Hom(S, S[i]) =\begin{cases}
				\bbC & \text{if}\; i = 0, \dim X,\\
				0 & \text{else},
			\end{cases}
		\end{align*}
		and $S \otimes \omega_X \cong S$.		
		An object $P \in \sfD^b(X)$ is called \emph{point-like} if
		\begin{align*}
			\Hom(P, P[i]) =\begin{cases}
				0 & \text{if}\; i < 0,\\
				\bbC & \text{if}\; i=0,
			\end{cases}
		\end{align*}
		and $P \otimes \omega_X \cong P$.
	\end{definition}
	Clearly, any shift of a skyscraper sheaf of a point $k(x)[m]$, $x\in X$, is a point-like object. By definition, a spherical object is a point-like object. 
	
	Denote by $p, q\colon X\times X \to X$ the projections and by $\Delta \hookrightarrow  X\times X$ the diagonal embedding.
	If $S$ is a spherical object on $X$, the object $\mcP_S \coloneqq \mathrm{Cone}(q^*S^\vee \otimes^{\mathrm{L}} p^* S \to \mcO_\Delta) \in \sfD^b(X\times X)$ is the Fourier--Mukai kernel of the \emph{spherical twist} $T_S \colon \sfD^b(X) \to \sfD^b(X)$ given by $T_S (-) = \mathrm{R}p_*(\mcP_S \otimes^{\mathrm{L}} q^*(-))$.
	By \cite[Thm.~1.2]{seidel_thomas_braid_group_actions_on_derived_categories_of_coherent_sheaves} a spherical twist is always an autoequivalence of $\sfD^b(X)$.
	
	The condition $P \otimes \omega_X \cong P$ has the following consequence on the support of a point-like object:
	\begin{lemma}\label{lem:spherical_obj_supported_on_curve}
		Let $X$ be a smooth projective positive dimensional variety with $K_X \neq 0$ in $\CH^*(X)_\bbQ$, i.e.\ $\omega_X$ is nontrivial and non-torsion.
		Then any point-like object $P \in \sfD^b(X)$ is supported on a connected proper closed subset.
		
		Moreover, if $X$ is a surface, then either
		\begin{enumerate}
			\item \label{item:case_1} $\Supp(P)$ is a point and $P \cong k(x)[m]$ for some $m \in \bbZ$, $x\in X$, or
			\item \label{item:case_2}$\Supp(P)$ is a, possibly reducible, connected curve $C=\bigcup_{i}C_{i}$ such that $K_X \vert_{ \tilde{C}_{i}} =0$ in $\CH^1(\tilde{C}_i)_\bbQ$, where $C_{i}$ are the irreducible components of $C$ and $\tilde{C}_i \to C_i$ are the normalizations.
		\end{enumerate}
		In particular, in \labelcref{item:case_1} $P$ is not spherical and in \labelcref{item:case_2} we have $K_X \cdot C =0 \in \bbZ$.
	\end{lemma}
	\begin{proof}
		Denote by $\mcH^i(P) \in \Coh X$ the $i$-th cohomology sheaf of $P$.
		Since $\omega_X$ is a line bundle, we have $$\mcH^i(P) \otimes  \omega_X = \mcH^i(P \otimes  \omega_X ) \cong \mcH^i(P),$$ which yields
		$\ch(\mcH^i(P)) \ch(\omega_X) = \ch (\mcH^i(P))$ in $\CH^*(X)_{\bbQ}$.
		If $\mcH^i(P)$ had positive rank, then $\ch(\mcH^i(P))$ would be invertible in $\CH^*(X)_{\bbQ}$, hence $\ch(\omega_X)=0$.
		This contradicts to $K_X$ being non-torsion.
		Hence, all cohomology sheaves $\mcH^i(P)$ have rank zero and thus the generic point of $X$ is not contained in the support of $P$.
		Thus, $\dim \Supp(P) < \dim X$ and $\Supp(P)$ is connected by \cite[Lem.~3.9]{huybrechts_fourier_mukai_transforms_in_algebraic_geometry}.
		
		Assume in addition that $\dim X = 2$.
		If $P$ is supported on a point, then \cite[Lem.~4.5]{huybrechts_fourier_mukai_transforms_in_algebraic_geometry} shows that $P \cong k (x)[m]$ for some $x\in X$ and $m \in \bbZ$.
		In particular, $\chi(k (x)[m], k (x)[m])=0$, so $P$ is not spherical.
		If $\Supp(P)$ is $1$-dimensional, $\Supp(P)$ is a connected reduced, possibly reducible, curve.
		
		Let $C_i \subseteq X$ be an irreducible curve, contained in $\Supp (P)$ and let $\tilde{C_i} \to C_i$ be its normalization.
		Denoting by $j\colon \tilde{C_i} \to C_i \hookrightarrow X$ composition, we obtain by the projection formula
		$$K_X \cdot C_i = j_*j^*K_X \in \CH_0(X).$$
		Let $\mcH$ be a cohomology sheaf of $P$ which has nonzero rank restricted on $C_i$.
		The equality $\ch(\mcH)=\ch(\mcH)\ch(\omega_X)$ on $X$ shows $\ch(j^*\mcH) = \ch(j^*\mcH)\ch(j^*\omega_X)$ on $\tilde{C_i}$. Since $j^*\mcH$ has nonzero rank, this implies that $j^* K_X$ is torsion in $\CH_0(\tilde{C_i})$.
		We conclude that the intersection number $K_X\cdot C_i = \deg (j_*j^*K_X)$ is zero.
	\end{proof}

	\section{Blow-up in $9$ Very General Points}\label{sec:9_points}
	Let $X\to \bbP^2_\bbC$ be the blow-up in the points $p_1, \dots, p_n \in \bbP^2$. We denote by $H \in \Pic(X)$ the class of the pullback of a hyperplane in $\bbP^2$ and by $E_i \in \Pic(X)$ the class of the $(-1)$-curve over the point~$p_i$. Recall that the canonical class of $X$ is given by $K_X = -3H + \sum_{i=1}^n E_i$ and that $\Pic(X)$ is freely generated by $H, E_1,\dots, E_n$. Moreover, the classes $H, E_1,\dots, E_n$ are pairwise orthogonal for the intersection pairing and satisfy $H^2=1$ and $E_i^2=-1$ for $1 \leq i \leq n$.
	The first proof of \cref{thm:main_thm} in the case of $n=9$ points builds on the following observation:
	\begin{proposition}\label{prop:restriction_to_cubic_non_torsion}
		Let $X$ be the blow-up of $\bbP^2_\bbC$ in $9$ very general points. Then $\lvert -K_X \rvert$ is zero-dimensional and the unique member is the strict transform $\tilde{C}$ of the unique smooth cubic curve $C \subseteq \bbP^2_\bbC$ which passes through the $9$ points. Moreover, $K_X \vert_{\tilde{C}} \neq 0$ in $\CH^1(\tilde{C})_\bbQ$.
	\end{proposition}
	\begin{proof}
		Let $U\subseteq (\bbP^2_\bbC)^9$ be the Zariski open subset parameterizing tuples of points $(p_1, \dots, p_9)$ such that $p_i \neq p_j$ for all $i\neq j$ and such that there is a unique smooth irreducible cubic $C \subseteq \bbP^2_\bbC$ passing through $p_1, \dots, p_9$.
		If $C\subseteq \bbP^2_\bbC$ is such a cubic, then $U \cap C^9$ is a nonempty Zariski open subset of $C^9$.
		Recall that for a fixed point $p_0 \in C$ the map $C \to \CH^1(C)$ sending $p \mapsto [p]-[p_0]$ is surjective.
		Thus, denoting $H$ for class of the hyperplane in $\bbP^2_\bbC$, the map $C^9 \to \CH^1(C)$ sending $(p_1, \dots, p_9) \mapsto 3H\vert_C - [p_1]-\dots -[p_9]$ is surjective.
		Since $\Pic^0(C)$ has only countably many torsion points, the set
		\begin{equation*}
			Z\coloneqq \left \{(p_1, \dots, p_9) \in U  \mid 3H\vert_C - \sum_{i=1}^9 [p_i] = 0 
			\; \text{in} \; \CH^1(C)_\bbQ \; \text{where}\; p_1, \dots, p_9 \in C\; \text{and}\; C \in  \lvert 3H \rvert \right \}
		\end{equation*}
		is a proper subset of $U$. To conclude the lemma, it is enough to show that $Z$ is a countable union of closed algebraic subsets of $U$. For this, we rewrite $Z$ as follows:
		Let
		\begin{align*}
			\mcC \coloneqq \left \{ (p_1, \dots, p_9, f, p) \in U \times \bbP(H^0(\bbP^2_\bbC, \mcO_{\bbP^2_\bbC}(3))) \times \bbP^2_\bbC \mid f(p_1) = \dots =f(p_9) =f(p)=0 \right \} \xrightarrow{\pi} U \\
			(p_1, \dots, p_9, f, p)  \mapsto (p_1, \dots, p_9)
		\end{align*}
		be the family of cubic curves through the $9$ points $(p_1, \dots, p_9)$.
		The map $\pi$ admits sections $s_i\colon U \to \mcC$ given by $(p_1, \dots, p_9) \mapsto (p_1, \dots, p_9, f, p_i)$ where $f$ is the unique cubic polynomial vanishing at $p_1, \dots, p_9$.
		We obtain cycles $\mathcal{P}_i \coloneqq [s_i(U)] \in \CH^1(\mcC)$ such that the restriction $\mathcal{P}_i \vert_{C_{\bar{p}}}$ to the fiber $C_{\bar{p}} = \pi^{-1}(\bar{p})$ over a point $\bar{p}= (p_1, \dots, p_9)$ is $[p_i] \in \CH^1(C_{\bar{p}})$.
		Let $\mcH \in \CH^1(\mcC)$ be the pullback of the hyperplane class on $\bbP^2_\bbC$.
		Then $3\mcH-\sum_i \mathcal{P}_i$ is a cycle in $\CH^1(\mcC)$ and the set $Z$ is the locus of points $\bar{p} \in U$ such that $(3\mcH-\sum_i \mathcal{P}_i)\vert_{C_{\bar{p}}}$ vanishes in $\CH^1(C_{\bar{p}})_\bbQ$. By \cite[Lem.~3.2]{voisin_chow_rings_decomp_of_diag_and_top_in_families}, $Z$ is a countable union of closed algebraic subvarieties.
	\end{proof}
	
	\begin{lemma}\label{lem:restriction_nonzero_for_all_curves}
		Let $X$ be the blow-up of $\bbP^2_\bbC$ in $n\leq 8$ general points or $n=9$ very general points. Then for any integral curve $C\subseteq \bbP^2_\bbC$, $K_X\vert_{\tilde{C}} \neq 0$ in $\CH^1(\tilde{C})_\bbQ$, where $\tilde{C}$ is the normalization of the strict transform of $C$.
	\end{lemma}
	\begin{proof}
		If $n\leq 8$, then $-K_X$ is ample and thus $-K_X \cdot \tilde{C}>0$ for every integral curve $\tilde{C} \subseteq X$. In particular, $ K_X\vert_{\tilde{C}} \neq 0$ since $\deg(-K_X\vert_{\tilde{C}}) = -K_X \cdot \tilde{C}>0$.
		
		If $n=9$, then $K_X^2=0$. It follows from the Hodge index theorem that the lattice $K_X^\perp \coloneqq \{v \in \Pic(X) \mid K_X \cdot v =0 \}$ is negative semi-definite. Furthermore, any class $v \in K_X^\perp$ with $v^2=0$ is of the form $v=mK_X$ for some $m\in \bbZ$.
		For $m>0$, $mK_X$ is not effective and for $m<0$ the linear system $\lvert mK_X \rvert$ is zero-dimensional, see, e.g., \cite[Thm.~2.5]{ciliberto_miranda_linear_systems_of_plane_curves_with_base_points_of_equal_mult}.
		Its unique member is the (multiple of the) strict transform of the unique smooth cubic $C$ through the $9$ blown up points.
		We conclude by \cref{prop:restriction_to_cubic_non_torsion} that $K_X\vert_{\tilde{C}} \neq 0$ in $\CH^1(\tilde{C})_\bbQ$. 
	\end{proof}
	
	With the above \cref{lem:restriction_nonzero_for_all_curves}, we can prove \cref{thm:main_thm} in case of $n=9$ points as follows:
	
	\begin{proof}[Proof of \cref{thm:main_thm}\,\labelcref{item:no_spherical_objects} for $n\leq 9$ points]
		Assume for contradiction that $S \in \sfD^b(X)$ is a spherical object.
		By \cref{lem:spherical_obj_supported_on_curve}, $S$ is supported on a connected curve $C = \bigcup_i C_i$ such that $K_X \vert_{C_i}=0 \in \CH^1(C_i)_\bbQ$ for all irreducible components $C_i$ of $C$.
		By \cref{lem:restriction_nonzero_for_all_curves} such curves $C_i$ do not exist.
	\end{proof}
	
	\begin{proof}[Proof of \cref{thm:main_thm}\,\labelcref{item:no_fm_partners,item:autoequvialences_are_standard} for $n\leq 9$ points]
		Let $\phi \colon \sfD^b(Y) \to \sfD^b(X)$ be an equivalence.
		For any point $y\in Y$ the skyscraper sheaf $k(y)$ satisfies $k(y) \otimes \omega_Y \cong k(y)$ and thus $\phi(k(y)) \otimes \omega_X \cong \phi(k(y))$.
		Since $\Hom (k(y), k(y)[i]) = \Hom(\phi(k(y)), \phi(k(y))[i])$ for all $i \in \bbZ$, $\phi(k(x))$ is a point-like object.
		By \cref{lem:spherical_obj_supported_on_curve}, $\Supp(\phi(k(y)))$ is either a point or $\Supp(\phi(k(y))) =\bigcup_i C_i$, where each $C_i$ is an integral curve with $K_X \vert_{C_i} = 0 \in \CH^1(C_i)_\bbQ$. By \cref{lem:restriction_nonzero_for_all_curves} such curves $C_i$ do not exist.
		Hence, $\phi(k(y))$ is supported on a point $x \in X$ and $\phi(k(y))= k(x)[m]$ for some $m \in \bbZ$.
		Moreover, by \cite[Cor.~6.14]{huybrechts_fourier_mukai_transforms_in_algebraic_geometry} the locus of $y\in Y$ such that $\phi \circ [-m](k(y))$ is a skyscraper sheaf is open.
		Since $Y$ is connected, this locus is the whole of $Y$, which shows that the shift $m$ in $\phi(k(y))=k(x)[m]$ is independent of $y\in Y$.
		Thus $\phi \circ [-m]$ sends skyscraper sheaves to skyscraper sheaves and \cite[\S\,3.3]{bridgeland_maciocia_fourier_mukai_transforms_for_quotient_varieties} (or \cite[Cor.~5.23]{huybrechts_fourier_mukai_transforms_in_algebraic_geometry}) shows that $\phi \circ [-m] = f_*(\mcL \otimes -)$ for some line bundle $\mcL \in \Pic(Y)$ and isomorphism $f \colon Y \to X$.
		This proves \labelcref{item:no_fm_partners} and shows that in the case $Y =X$ the autoequivalence $\phi$ is a standard autoequivalence. Thus, \labelcref{item:autoequvialences_are_standard} follows.
	\end{proof}
	
	\begin{remark}
		If one could prove the conclusion of \cref{lem:restriction_nonzero_for_all_curves} for blow-ups $X$ of $\bbP^2_\bbC$ in $n>9$ very general points, then \cref{thm:main_thm} would follow by the same arguments as used above in the case of $n\leq 9$ points.
	\end{remark}

	\section{Proof of the General Case}\label{sec:general_case}
	In the general case of \cref{thm:main_thm}\,\labelcref{item:autoequvialences_are_standard,item:no_fm_partners} follow from the work of Uehara \cite[Thm.~1.1, Thm.~1.3]{uehara_a_trichotomy_for_the_autoequivalence_groups_on_smooth_projective_surfaces}, Kawamata \cite{kawamata_d_equivalence_and_k_equivalence}, and the following result of de~Fernex:
	\begin{proposition}[{\cite[Prop.~2.4]{de_fernex_negative_curves_on_very_general_blow_ups_of_p2}}]\label{prop:negative_curves_de_fernex}
		Let $X$ be the blow-up of $\bbP^2_\bbC$ in a finite set of points in very general position.
		If $C\subseteq X$ is an integral rational curve with $C^2<0$, then $C$ is a $(-1)$-curve, that is a smooth rational curve of self-intersection $-1$.
	\end{proposition}
	
	\begin{proof}[Proof of \cref{thm:main_thm}\,\labelcref{item:autoequvialences_are_standard,item:no_fm_partners}]

		Recall, e.g.\ from \cite[Prop.~2.2]{cantant_dolgachev_rational_surfaces_with_a_large_group_of_automorphisms}, that if $Y$ is a rational surface admitting a minimal elliptic fibration, then $Y$ can be obtained from $\bbP^2_\bbC$ by blowing up $9$, possibly infinitely near, points and, for some $m>0$, the linear system $\lvert -mK_Y\rvert$ is a pencil.
		Hence, if $X$ is the blow-up of $\bbP^2_\bbC$ in a finite set of points in very general position, then $X$ admits no minimal elliptic fibration.
		Indeed, this is clear if the number of blown up points is different from $9$.
		In the case of $9$ blown up points the linear system $\lvert -mK_X\rvert$ is zero-dimensional for any $m>0$, so it is not a pencil.
		By \cite[Thm.~1.6]{kawamata_d_equivalence_and_k_equivalence}, a non-minimal surface admits nontrivial Fourier--Mukai partners only if it admits a minimal elliptic fibration. Hence, \cref{thm:main_thm}\,\labelcref{item:no_fm_partners} follows.
		
		Let $Y$ be any surface and let $\Phi_P \colon \sfD^b(Y) \to \sfD^b(Y)$ be an autoequivalence with Fourier--Mukai kernel $P\in \sfD^b(Y\times Y)$.
		We denote by $\mathrm{Comp}(\Phi_P)$ the set of irreducible components of $\Supp(P) \hookrightarrow Y \times Y$ and by
		$$N_Y \coloneqq \max \{\dim W \mid W \in \mathrm{Comp}(\Phi_P), \Phi_P\in \Aut(\sfD^b(Y))\}$$
		the \emph{Fourier–Mukai support dimension} of $Y$.
		By Uehara's classification \cite[Thm.~1.1]{uehara_a_trichotomy_for_the_autoequivalence_groups_on_smooth_projective_surfaces}, the equality $N_Y=2$ is equivalent to $Y$ admitting no minimal elliptic fibration and $K_Y$ being not numerically equivalent to zero.
		Hence, for $X$ the blow-up of $\bbP^2_\bbC$ in a finite set of points in very general position we have $N_X=2$.
		
		If $Y$ is a surface with $N_Y = 2$ such that the union of all $(-2)$-curves in $Y$ forms a disjoint union of configurations of type $A$, then, by \cite[Thm.~1.3]{uehara_a_trichotomy_for_the_autoequivalence_groups_on_smooth_projective_surfaces}, $\Aut(\sfD^b(Y))$ is generated by standard autoequivalences and spherical twists.
		For $X$ the blow-up of $\bbP^2_\bbC$ in a finite set of points in very general position, de~Fernex' \cref{prop:negative_curves_de_fernex} shows that $X$ contains no $(-2)$-curve.
		Hence, \cref{thm:main_thm}\,\labelcref{item:autoequvialences_are_standard} follows.     
	\end{proof}
	
	\section{Surfaces of Nonnegative Kodaira Dimension}
	\label{se:surfaces_nonneg_kod_dim}
	\subsection{Autoequivalences}
	In contrast to the case of negative Kodaira dimension, blowing up points in arbitrary position on minimal surfaces of nonnegative Kodaira dimension does not give rise to new $(-2)$-curves.
	\begin{proposition}\label{prop:-2_curves_obj_in_blow_up_nef_canonical}
		Let $Y$ be a minimal surface of nonnegative Kodaira dimension and let $p \colon X \to Y$ be the blow-up of $Y$ in a set of points $p_1, \dots, p_n \in Y$.
		Then every $(-2)$-curve $C$ in $X$ is the strict transform of a $(-2)$-curve $C_0$ in $Y$ such that $p_i \notin C_0$ for $1\leq i \leq n$.
	\end{proposition}
	\begin{proof}
		We denote by $E_i$ the exceptional divisor over the $i$-th blown up point $p_i$ for $1 \leq i \leq n$.
		Let $C \subseteq X$ be a $(-2)$-curve. By adjunction, we have
		$$0=g(C) = 1 + \frac{1}{2}(C^2 + C \cdot K_X),$$
		where $g(C)$ denotes the geometric genus of $C$.
		Thus, $C \cdot K_X =0$.
		Further, since $C$ is not one of the exceptional curves $E_i$, $C$ is the strict transform of a curve $C_0 \subseteq Y$.
		We have
		$$0 = C \cdot K_X  = C_0 \cdot K_Y + \sum_{i=1}^n m_i,$$
		where $m_i$ is the multiplicity of $C_0$ at $p_i$.
		Since $K_Y$ is nef, each of the $m_i$ is zero, in other words $p_i \notin C_0$ for $1 \leq i \leq n$.
		We conclude that $C_0$ is a smooth rational curve with $K_Y \cdot C_0=0$, hence, by adjunction, a $(-2)$-curve.
	\end{proof}
	As a consequence of \cref{prop:-2_curves_obj_in_blow_up_nef_canonical} and \cite{uehara_a_trichotomy_for_the_autoequivalence_groups_on_smooth_projective_surfaces}, we obtain the following
	\begin{theorem}\label{thm:short_version_blow_up_nonneg_kod_dim}
		Let $Y$ be a minimal
		surface of 
		nonnegative Kodaira dimension and let $X$ be the blow-up of $Y$ in a nonempty finite set of points.
		Assume $Y$ contains no $(-2)$-curves, e.g.\ $Y$ has Kodaira dimension $1$ and the elliptic fibration of $Y$ has only irreducible fibers.
		Then $\sfD^b(X)$ admits only standard autoequivalences, i.e.
		$$\Aut(\sfD^b(X)) =  \Pic(X) \rtimes \Aut(X) \times \bbZ[1].$$
	\end{theorem}
	\begin{proof}
		By \cref{prop:-2_curves_obj_in_blow_up_nef_canonical}, $X$ contains no $(-2)$-curves.
		Thus, the statement follows from \cite[Thm.~1.1, Thm.~1.3]{uehara_a_trichotomy_for_the_autoequivalence_groups_on_smooth_projective_surfaces} if $X$ admits no minimal elliptic fibration. The latter can be shown as follows:        
		Recall, e.g.\ from \cite[Cor.~4.1.7]{cossec_dolgachev_liedtke_enriques_i}, that a surface $S$ with minimal elliptic fibration satisfies $K_S^2=0$.
		If $\kappa(Y)=0$, then $K_Y$ is numerically equivalent to zero. Hence, $K_Y^2=0$ and therefore $K_X^2<0$.
		If $\kappa(Y)=1$, then $Y$ has an elliptic fibration and therefore $K_Y^2=0$. Hence, $K_X^2<0$.
		Finally, if $\kappa(Y)=2$, then $X$ has no elliptic fibration by \cite[Ch.~V, Prop.~12.5]{barth_hulek_peters_van_de_ven_compact_complex_surface}.
	\end{proof}
	
	\begin{remark}
		Note that the description of autoequivalences as in \cref{thm:short_version_blow_up_nonneg_kod_dim} is not true for a minimal surface $Y$.
		For example, if $\kappa(Y)= 1$, then $\Aut(\sfD^b(Y))$ can be characterized as in \cite[Thm.~4.1]{uehara_autoequvialences_of_derived_categories_of_elliptic_surfaces_with_non_zero_kodaira_dimension}.
		In that case, 
		as outlined in the proof of \cite[Thm.~1.1]{uehara_a_trichotomy_for_the_autoequivalence_groups_on_smooth_projective_surfaces}, $Y$ admits an autoequivalence $\Phi_\mcU$ where $\mcU$ is the universal sheaf on $Y \times J_Y(1,1)$ and $J_Y(1,1)\cong Y$ is a moduli space of stable sheaves on a smooth fiber of the elliptic fibration of $Y$.
		In this case, the support of $\mcU$ is $3$-dimensional, thus $\Phi_\mcU$ does not lift to an autoequivalence of a blow-up of $Y$.
	\end{remark}
	\begin{remark}[Infinitely near points]
		Let $X$ be a non-minimal surface of nonnegative Kodaira dimension with minimal model $Y$. If the $(-2)$-curves in $Y$ only form chains of type~$A$, then it is possible to describe $\Aut(\sfD^b(X))$ as in \cite[Thm.~1.3]{uehara_a_trichotomy_for_the_autoequivalence_groups_on_smooth_projective_surfaces}.
		Indeed, arguing as in \cite[Thm.~1.5]{ishii_uehara_autoequivalences_of_derived_categories_on_the_minimal_resolutions_of_an_singularities_on_surfaces} one shows that the $(-2)$-curves in $X$ only form chains of type~$A$.
		Thus, \cite[Thm.~1.3]{uehara_a_trichotomy_for_the_autoequivalence_groups_on_smooth_projective_surfaces} applies and shows that $\Aut(\sfD^b(X))$ is generated by standard autoequivalences and spherical twists.
	\end{remark}
	
	\subsection{Spherical Objects}
	Similar to \cref{prop:-2_curves_obj_in_blow_up_nef_canonical}, spherical objects in the blow-up of a minimal surface of positive Kodaira dimension are completely determined by the minimal surface.
	
	We begin with recalling two elementary \cref{lem:disjoint_support,lem:characterising_support_by_morphisms} regarding morphisms and the support of complexes of sheaves.
	As we were unable to find a suitable statement in the literature, we include a proof of \cref{lem:disjoint_support}.
	\begin{lemma}\label{lem:disjoint_support}
		Let $X$ be a smooth projective variety and let $F, G\in \sfD^b(X)$.
		\begin{enumerate}
			\item \label{item:first_item}If $\Supp(F) \cap \Supp(G) = \emptyset$, then $\Hom_{\sfD^b(X)}(F, G) =0$.
			\item \label{item:second_item} If $D \subseteq X$ is a divisor and $\Supp(F) \cap D = \emptyset$, then $F \otimes \mcO_X(D) =  F$.
		\end{enumerate}
	\end{lemma}
	\begin{proof}
		We first prove \labelcref{item:first_item}.
		The condition $\Supp(F) \cap \Supp(G) = \emptyset$ implies $\sheafext_{\mcO_X}^p(\mcH^{-q}(F), \mcH^l(G)) =0$ for all $p,q,l \in \bbZ$.
		Recall, e.g.\ from \cite[p.~77]{huybrechts_fourier_mukai_transforms_in_algebraic_geometry}, that we have a spectral sequence
		$$E_2^{p,q} = \sheafext_{\mcO_X}^p(\mcH^{-q}(F), \mcH^l(G)) \Rightarrow \sheafext_{\mcO_X}^{p+q}(F, \mcH^l(G))$$
		for every $l \in \bbZ$.
		Similarly, we have a spectral sequence
		$$E_2^{p,q} = \sheafext_{\mcO_X}^p(F, \mcH^q(G)) \Rightarrow \sheafext_{\mcO_X}^{p+q}(F, G).$$
		Thus, $\Supp(F) \cap \Supp(G) = \emptyset$ implies $\sheafext_{\mcO_X}^{l}(F, G) =0$ for all $l \in \bbZ$.
		Finally, the local-to-global spectral sequence
		$$E^{p,q}_2 = H^p(X, \sheafext_{\mcO_X}^q(F, G)) \Rightarrow \Ext_{\mcO_X}^{p+q}(F, G)$$
		shows $\Ext_{\mcO_X}^{p+q}(F, G)=0$.
		
		To prove \labelcref{item:second_item}, assume that $D \subseteq X$ is a divisor and that $\Supp(F) \cap D = \emptyset$.
		The ideal sheaf sequence
		$$0 \to \mcO_X(-D) \to \mcO_X \to \mcO_D \to 0$$ yields an exact sequence
		$$0 \to \sheafhom_{\mcO_X}(\mcO_D, F) \to F \to F \otimes \mcO_X(D) \to \sheafext_{\mcO_X}^1(\mcO_D, F) \to 0.$$
		As argued above, we have $\sheafhom_{\mcO_X}(\mcO_D, F) =0 =\sheafext_{\mcO_X}^1(\mcO_D, F)$. Hence, $F \to F \otimes \mcO_X(D)$ is an isomorphism.
	\end{proof}
	
	\begin{lemma}[{\cite[Lem.~5.3]{bridgeland_maciocia_fm_trsnforms_for_k3_and_elliptic_fibrations}}]\label{lem:characterising_support_by_morphisms}
		Let $X$ be a smooth projective variety and $F \in \sfD^b(X)$. Then a point $x \in X$ lies in $\Supp(F)$ if and only if $\Hom_{\sfD^b(X)}(F, k(x)[l]) \neq 0$ for some $l \in \bbZ$.
	\end{lemma}
	
	The following \cref{prop:spherical_obj_in_blow_up_nef_canonical} characterizes spherical objects in blow-ups of minimal surfaces of positive Kodaira dimension.
	\begin{proposition}\label{prop:spherical_obj_in_blow_up_nef_canonical}
		Let $Y$ be a minimal surface of positive Kodaira dimension and let $p \colon X \to Y$ be the blow-up of $Y$ in a set of points $p_1, \dots, p_n \in Y$.
		Then every spherical object in $\sfD^b(X)$ is of the form $\mathrm{L}p^*S$ for some spherical object $S \in \sfD^b(Y)$.
		Moreover, if $S \in \sfD^b(Y)$ is spherical, then $\mathrm{L}p^*S$ is spherical if and only if $p_i \notin \Supp(S)$ for all $1 \leq i \leq n$.
	\end{proposition}
	\begin{proof}
		We denote by $E_i$ the exceptional divisor over the $i$-th blown up point $p_i$ for $1 \leq i \leq n$.
		We first prove the following
		\begin{claim}
			If $S' \in \sfD^b(X)$ is a spherical object, then $\Supp(S')$ is disjoint from each $E_i$.
		\end{claim}
		\begin{proof}[Proof of the claim.]
			Assume $S' \in \sfD^b(X)$ is spherical, then, by \cref{lem:spherical_obj_supported_on_curve}, $\Supp(S') = \bigcup_i C_i$, where each $C_i$ is an integral curve with $K_X \cdot C_i =0$.
			Since $K_X = p^*K_Y + \sum_i E_i$, such curve $C_i$ is the strict transform of a curve in $Y$.
			Moreover, if $C_0$ is a curve in $Y$, the strict transform of $C_0$ has class $p^*C_0 - \sum_i m_i E_i$, where $m_i$ is the multiplicity of $C_0$ at $p_i$.
			We compute that
			$$K_X \cdot \left(p^*C_0 - \sum_{i=1}^n m_i E_i \right) = K_Y \cdot C_0 + \sum_{i=1}^n m_i.$$
			Since $K_Y$ is nef, we have $K_Y \cdot C_0 \geq 0$ and therefore $p_i \notin C_0$ for all $1\leq i\leq n$.
		\end{proof}                  
		
		Recall that $\sfD^b(X)$ admits a semiorthogonal decomposition
		$$\sfD^b(X) = \langle \mcO_{E_1}(-1), \dots, \mcO_{E_n}(-1), \mathrm{L}p^* \sfD^b(Y) \rangle.$$
		Since $\Supp(S')$ is disjoint from each $E_i$, we have, by \cref{lem:disjoint_support},
		$$\Hom_{\sfD^b(X)} (S', \mcO_{E_i}(-1)[l])=0=\Hom_{\sfD^b(X)} ( \mcO_{E_i}(-1), S'[l])$$
		for every $l \in \bbZ$.
		Hence, $S' \in \mathrm{L}p^* \sfD^b(Y)$, i.e., there exists a object $S \in \sfD^b(Y)$ such that $\mathrm{L}p^{*}S\cong S'$.
		Note that $\mathrm{R}p_*\mcO_X = \mcO_Y$ implies
		\begin{align}\label{eq:Hom_sequence}
			& \Hom_{\sfD^{b}(X)}(S',S'[l])=\Hom_{\sfD^b(X)}(\mathrm{L}p^* S, \mathrm{L}p^* S[l]) = \Hom_{\sfD^b(Y)}( S,\mathrm{R}p_* \mathrm{L}p^* S[l]) \\
			={}& \Hom_{\sfD^b(Y)}( S, S\otimes^\mathrm{L} \mathrm{R}p_*\mcO_X[l]) = \Hom_{\sfD^b(Y)}( S, S[l]). \nonumber
		\end{align}
		for every $l\in \bbZ$.
		Moreover, since $\Supp(S')$ is disjoint from the exceptional divisors $E_{i}$, \cref{lem:disjoint_support} shows that $\mathrm{L}p^{*}S\otimes \mcO_{X}(\sum_{i}E_{i}) = \mathrm{L}p^{*}S$. Hence, $\mathrm{L}p^{*}S \otimes p^* \omega_Y \cong \mathrm{L}p^{*}S$.
		Pushing forward via $\mathrm{R}p_*$ and using the projection formula shows that $S \otimes \omega_Y \cong S$. Thus, $S$ is a spherical object in $\sfD^b(Y)$.
		
		Now let $S \in \sfD^b(Y)$ be a spherical object. 
		As in \cref{eq:Hom_sequence}, we have
		\begin{align*}
			& \Hom_{\sfD^b(X)}(\mathrm{L}p^* S, \mathrm{L}p^* S[l]) = \Hom_{\sfD^b(Y)}( S, S[l]) 
		\end{align*}
		for every $l \in \bbZ$.
		Thus, $\mathrm{L}p^* S$ is spherical if $\mathrm{L}p^*S \otimes \omega_X \cong \mathrm{L}p^*S$.
		Let $x \in X$ be a point, then $\mathrm{R}p_* k(x)=k(p(x))$ and by adjunction
		$$\Hom_{\sfD^{b}(X)}( S, \mathrm{R}p_* k(x)[l]) = \Hom_{\sfD^{b}(X)}(\mathrm{L}p^*S, k(x)[l])$$
		for every $l \in \bbZ$.
		Hence, \cref{lem:characterising_support_by_morphisms} shows that  $\Supp(\mathrm{L}p^* S ) = p^{-1}(\Supp(S))$.
		By the previous claim, 
		it is necessary that $p^{-1}(\Supp(S))$ is disjoint from each $E_i$
		for $\mathrm{L}p^* S$ to be spherical.
		On the other hand, this is also sufficient, since $\mathrm{L}p^* S \otimes \mcO_X(\sum_i E_i)= \mathrm{L}p^* S$ holds by \cref{lem:disjoint_support} if $p^{-1}(\Supp(S))$ is disjoint from each $E_i$.
	\end{proof}
	
	\begin{remark}\label{rmk:spherical_objects_base_no_spherical_nef_canonical}
		Let $Y$ be a minimal surface of Kodaira dimension $1$ whose elliptic fibration has only irreducible fibers.
		It follows from the description of $\Aut(\sfD^b(Y))$ in \cite{uehara_autoequvialences_of_derived_categories_of_elliptic_surfaces_with_non_zero_kodaira_dimension} that $\sfD^b(Y)$ does not contain spherical objects.
		Thus, if $X$ is a blow-up of $Y$ in a finite set of points, then, by \cref{prop:spherical_obj_in_blow_up_nef_canonical}, $\sfD^b(X)$ does not contain spherical objects either.
		Alternately, this can also be deduced from \cref{thm:short_version_blow_up_nonneg_kod_dim}.
	\end{remark}

	\printbibliography

\end{document}